\documentclass[12pt]{article}
\usepackage[left=1in,top=1in,right=1in,bottom=1in,nohead]{geometry}

\newcount\colveccount
\newcommand*\colvec[1]{
        \global\colveccount#1
        \begin{pmatrix}
        \colvecnext
}
\def\colvecnext#1{
        #1
        \global\advance\colveccount-1
        \ifnum\colveccount>0
                \\
                \expandafter\colvecnext
        \else
                \end{pmatrix}
        \fi
}
\usepackage{url,subfig,latexsym,amsmath,amssymb, amsfonts,enumerate,
  epsfig, graphics,times, amsthm,mathtools, bbm}
\usepackage{color}
\usepackage{nameref}
\usepackage{listings}
\usepackage[colorlinks]{hyperref}
\usepackage{makeidx}
\usepackage{tabularx}
\usepackage{array}
    \newcolumntype{P}[1]{>{\centering\arraybackslash}p{#1}}
    \newcolumntype{M}[1]{>{\centering\arraybackslash}m{#1}}
\newtheorem{thm}{Theorem}[section]
\newtheorem{cor}[thm]{Corollary}
\newtheorem{lem}[thm]{Lemma}

\newtheorem*{thm*}{Theorem}

\theoremstyle{definition}
\newtheorem{defn}{Definition}[section]
\newtheorem{rem}{Remark}[section]
\newtheorem{example}{Example}[section]

\newcommand{\thmref}[1]{Theorem~\ref{#1}}

\newcommand{\dlim}{\displaystyle \lim\limits}

\def\T1{T^{D,1}_{\{x_n\}}}
\def\S{\mathcal S}
\def\C{\mathcal C}
\def\Re{\mathcal R}
\def\R{\mathbb R}

\def\Z{\mathbb Z}
\def\A{\mathcal A}
\def\Td1{T^{D,1}_{\{x_n\}}}
\def\Ts1{T^{S,1}_{\{x_n\}}}

\def\TV{\text{TV}}

\def\lmt{\lim_{n\rightarrow \infty}}
\def\K{\mathcal{K}}

\def\Tid1{T^{D,1}_{\{x_n+\zeta_i\}}}
\def\Tis1{T^{S,1}_{\{x_n+\zeta_i\}}}
\def\Tjs1{T^{S,1}_{\{x_n+\zeta_j\}}}
\def\Tjd1{T^{D,1}_{\{x_n+\zeta_j\}}}

\makeatletter
\newcommand{\labitem}[2]{%
\def\@itemlabel{\textbf{#1}}
\item
\def\@currentlabel{#1}\label{#2}}
\makeatother

\title{
Mixing times for two classes of stochastically modeled reaction networks}

\author{
David F. Anderson\thanks{Department of Mathematics, University of
  Wisconsin, Madison, USA.  anderson@math.wisc.edu.},
 \and
Jinsu Kim\thanks{Department of Mathematics, Pohang University of Science Technology, Pohang 37673, Republic of Korea. jinsukim@postech.ac.kr, grant support from the National Research Foundation of Korea (NRF). Corresponding author.  }
  }

\begin{document}

\maketitle
\begin{abstract}
The past few decades have seen robust research on questions regarding the existence, form, and properties of stationary distributions of stochastically modeled reaction networks.  When a stochastic model admits a stationary distribution an important practical question is: what is the rate of convergence of the distribution of the process to the stationary distribution? With the exception of \cite{XuHansenWiuf2022} pertaining to models whose state space is restricted to the non-negative integers, there has been a notable lack of results related to this rate of convergence in the reaction network literature.  This paper begins the process of filling that hole in our understanding. In this paper, we characterize this rate of convergence, via the mixing times of the processes, for two classes of stochastically modeled reaction networks.  Specifically, by applying a Foster-Lyapunov criteria we  establish exponential ergodicity for two  classes of reaction networks introduced in \cite{anderson2018some}.   Moreover, we show that for one of the classes the convergence is uniform over the initial state.
\end{abstract}



\section{Introduction}

Stochastic models of reaction networks, which will be formally introduced in Section \ref{sec:RN},  are used ubiquitously in the biosciences to track the (integer valued) counts of different interacting species.  This class of models is utilized at multiple different scales, from molecular processes to the level of populations.  
Numerous previous works have considered questions related to the stationary behavior of these models (see \cite{AM202,anderson2020stochastically, anderson2016product, AndProdForm, anderson2018some, Othmer2005, hoessly2020algebraic, XuHansenWiuf2022} for a subset).  In this paper, we study the natural followup question: if $\pi$ is the stationary distribution of the model, how fast does $P^t(x,\cdot)$ converge to $\pi(\cdot)$, where $P^t(x,A) = P(X(t) \in A| X(0) = x)$ is the time-dependent distribution of the process, denoted by $X$, given an initial condition of $X(0) = x$?

A natural route forward is to consider the \textit{mixing time} of the process.  Let $\varepsilon \in (0,\tfrac12)$.  For a Markov process $X$ with a stationary distribution $\pi$,  the \textit{mixing time} for the process with initial condition $x$ is
\begin{align}\label{eq:mixing time}
\tau_x^\varepsilon = \inf_{t \ge 0} \{\Vert P^t(x,\cdot) - \pi(\cdot) \Vert_{\TV} \le \varepsilon\}, 
\end{align}
where the total variation distance between two probability measures on a measurable space $(\Omega,\mathcal{F})$ is defined as $\Vert \mu - \nu\Vert_{\TV} = \sup_{A\in \mathcal{F}}|\mu(A) - \nu(A)|$.  In our case of a discrete state space, $\|\mu-\nu\|_{\TV} = \frac12 \sum_x |\mu(x) - \nu(x)|$, and so the total variation norm is an $L^1$ norm.  See \cite{YuvalLevinMixing} for a nice introduction to mixing times.

In Theorems \ref{thm:single_mixing} and \ref{thm:doublefull mixing}, we will provide results on the mixing times for two distinct classes of reaction networks, with each class characterized by different graph topological conditions on the associated reaction network. The network types considered in this paper were first considered in \cite{anderson2018some},  where positive recurrence (and hence existence of a stationary distribution) was established for each.  Thus, this paper can  naturally be viewed as a followup to \cite{anderson2018some}. For the sake of consistency, throughout this manuscript we used the same format and wordings of the key definitions and theorems as \cite{anderson2018some}.

For now, we simply write $\text{Cond}_1$ and $\text{Cond}_2$  for the sets of conditions considered in Theorems \ref{thm:single_mixing} and \ref{thm:doublefull mixing}, respectively (the specific conditions will be detailed in Section \ref{sec:mainresults} when we have the proper terminology available).    Each of our theorems is then of the following form:  for a stochastically modeled reaction network whose associated graph satisfies the conditions $\text{Cond}_i$ and whose stationary distribution is $\pi_i$, there is an $\eta_i >0$ and a function $B_i:\Z^d_{\ge 0} \to \R_{\ge0}$ for which
\begin{align}\label{eq:toprove}
\Vert P_i^t(x,\cdot) - \pi_i(\cdot) \Vert_{\TV} \le B_i(x) e^{-\eta_i t}, \quad \text{for all } t \ge 0.
\end{align}
Hence, $\tau_x^\varepsilon \le \frac1\eta\left[ \ln(B_i(x) + \ln(\varepsilon^{-1})\right] = O(\ln(B_i(x)).$  The exponential decay rate \eqref{eq:toprove} with $\eta_i>0$ means that the Markov process $X$ is \emph{exponentially ergodic}. Moreover, we characterize $B_i(x)$ in each case with $B_1(x) = C_1 (|x| +1)\ln(|x| + 2)$ and $B_2(x) = C_2$, for some $C_1,C_2>0$.  Hence, the convergence for the second class is shown to  be uniform over the initial condition. 

Our proofs rely upon the use of Foster-Lyapunov functions.  Specifically, for  Theorem \ref{thm:single_mixing} we will provide an appropriate function, $V:\Z^d_{\ge 0} \to \R_{\ge 0}$, with $V(x)\to \infty$ as $|x|\to \infty$, and constants $a>0$ and $b>0$ so that $\mathcal{A} V(x) \le -a V(x) + b$ for all $x$ outside some compact set, where $\mathcal{A}$ is the generator of the process.  For Theorem \ref{thm:doublefull mixing} we will provide an appropriate function $V$, and positive constants $a,b$, and $\delta$ so that $\mathcal{A} V(x) \le -a V(x)^{1+\delta} + b$ for $x$ outside a compact. The (known) results of Section \ref{sec:preliminary} related to Foster-Lyapunov functions are then utilized to complete the argument. The new results presented here should be compared with the results of \cite{anderson2018some} in which positive recurrence was proven by demonstrating the existence of an appropriate function $V$ with $\mathcal{A} V(x) \le -1$ for all $x$ outside some compact set.  The sharper estimates provided here are a result of a more careful analysis of the generator.

The remainder of the paper is outlined as follows.  In Section \ref{sec:RN}, we introduce stochastic models of reaction networks.  In Section \ref{sec:mainresults}, we precisely state our main results.  In Section \ref{sec:preliminary}, we provide some necessary results related to tiers (in the context of reaction networks) and Foster-Lyapunov functions.   In Section \ref{sec:proofs}, we provide proofs for our  main results using appropriate Foster-Lyapunov functions.  In Section \ref{sec:generalizations}, we provide some generalizations to our main results.

\section{Stochastic reaction networks}
\label{sec:RN}

 Stochastic reaction networks are a class of continuous-time Markov chains on $\Z^d_{\ge 0}$ that are used to model  biochemical systems and other population processes \cite{AndKurtz2011,AK2015,Gill76,Wilkinson2006}. Of particular interest to mathematical  researchers is how the  topological conditions of the underlying reaction graph relates to the qualitative behavior of the associated dynamical system (such as the mixing time).
  We proceed with the basic definitions from the field.

A reaction network is a graphical construct that describes a  set of possible interactions among some constituent ``species.'' 
\begin{defn}\label{def:21}
\emph{A  reaction network} is given by a triple of finite sets $(\S,\C,\Re)$ where:
\begin{enumerate}[(i)]
\item \emph{The species set} $\S=\{S_1,S_2,\cdots,S_d\}$, with $d<\infty$,  contains the species of the reaction network.
\item \emph{The reaction set} $\Re=\{R_1,R_2,\cdots,R_r\}$, with $r< \infty$, consists of ordered pairs $(y,y') \in \Re$ where 
\begin{align}\label{complex}
y=\sum_{i=1}^d y_iS_i \hspace{0.4cm} \textrm{and} \hspace{0.4cm}
y'=\sum_{i=1}^d y'_iS_i,
\end{align}
and where the values $y_i,y'_i \in \mathbb{Z}_{\ge 0}$ are the \emph{stoichiometric coefficients}. We will often write reactions $(y,y')$ as $y\rightarrow y'$.  

\item \emph{The complex set} $\C$ consists of the linear combinations of the species in (\ref{complex}). Specifically, 
$\C = \{y\ |\ y\rightarrow y' \in \Re\} \cup \{y' \ |\ y\rightarrow y' \in \Re\}$. For the reaction $y\to y'$, the complexes $y$ and $y'$ are termed the \emph{source} and \emph{product} complex of the reaction, respectively. \hfill $\square$
\end{enumerate}
\end{defn}

The state space of our eventual Markov model, which will be fully described after some discussion about reaction networks, will be described by
the copy numbers of the species.  The transitions of the model will be determined via the reactions.

We note that in particular examples, as in most of ours below, we do not enumerate the species as $S_1,\dots, S_d$ and instead utilized more suggestive notation (e.g., $E$ for  Enzyme, $P$ for a protein, etc.), or alphabetical labeling (e.g., $A$, $B$, $C$, etc.).  The clear enumeration of $\{S_1,S_2,\dots, S_d\}$ will be used primarily for theory  and proofs.  

The linear combinations in \eqref{complex} are, for example,  of the form $2S_1 + S_2$ or $2S_1 + S_2 + 4S_3$.  However, we will also associate each complex $y$ with  the vector whose $j$-th component is $y_j$,  i.e.~$y=(y_1,y_2,\cdots,y_d)^T \in \mathbb{Z}^d_{\ge 0}$. 
For example, when $\S=\{S_1,S_2,\dots,S_d\}$, the complex $2S_1+S_2$ is associated with the vector $(2,1,0,0,\dots,0)^T \in \Z^d_{\ge 0}$. 
We further note that it is reasonable to consider a complex $y$ with $y_i=0$ for each $i$. This complex is denoted by $\emptyset$.

We sometimes enumerate complexes and/or reactions. Hence, we will sometimes write, for example,   $y_1$ or $y_r$   to denote particular complexes, as opposed to the $1$st or $r$th element of the vector $y$.  Similarly, we may refer to the $r$th reaction as $y_r \to y_r' \in \Re$.  In such  cases, the $j$th coordinate of, for example,  $y_1$ and $y_r$ will be denoted $y_{1,j}$ and $y_{r,j}$, respectively.  Context will always make it clear if $y_r$ is referring to the $r$th element of a complex $y$ or if $y_r$ denotes a particular  complex (such as the source complex of the reaction $y_r\to y_r'$).

It is most common to present a reaction network with a directed graph, termed the \textit{reaction graph} of the network, in which the vertices are the complexes, each complex is written exactly one time (even if it appears in more than one reaction), and the directed edges are given by the reactions.
We present an example to solidify notation.

\begin{example}
The following reaction graph is associated with a reaction network modeling a usual substrate-enzyme kinetics

\[
	S+E \rightleftarrows SE \rightarrow E+P.
\]
 For this reaction network, $\S = \{S,E,SE,P\}$, $\C=\{S+E,SE,E+P\}$ and $\Re=\{S+E\rightarrow SE, SE\rightarrow S+E,SE\rightarrow E+P\}$.   \hfill $\triangle$
\end{example}

It is conditions on the reaction graph that will be important to us.  It is unfortunately the case that some of the terminology used in the study of reaction networks is different than that used more commonly.  Hence, we discuss here some of the differences.

A subgraph, which is connected, of the reaction graph is termed a \emph{linkage class \cite{clark1991first}.  If a linkage class is strongly connected} then we say that linkage class is \emph{weakly reversible}.  Finally, we say that the entire network/graph is \emph{weakly reversible} if each connected component (linkage class) is weakly reversible.

\begin{example}\label{ex11}
If we consider a reaction network described with the reaction graph,
\begin{align*}
&A \rightarrow 2B \rightleftarrows A+D, \qquad \emptyset \rightarrow B, \qquad B+C \rightarrow 2C,\\[-1ex]
&  \hspace{6.8cm}  \displaystyle \nwarrow \hspace{.2in}  \swarrow \\[-1ex]
& \hspace{2.9in} D
\end{align*}
then one can see that there are four species, eight complexes (vertices), seven reactions (directed edges), and three linkage classes (connected components).  The right-most connected component is strongly connected, and hence weakly reversible, whereas the other two are not.  Hence, the network as a whole is \emph{not} said to be weakly reversible.
\hfill $\triangle$
\end{example}


 The following notation  will be utilized.
 \begin{enumerate}
 \item For $u,v \in \mathbb{R}^d_{\ge 0}$, we define $u^v=\prod_{i=1}^d u_i^{v_i}$, with $0^0$ taken to be equal to 1. 
 \item For $x \in \mathbb{Z}^d_{\ge 0}$, we let $x! = \prod_{i=1}^d x_i!$.
 \end{enumerate}

We now show the most common way to stochastically model a reaction network.
Given a reaction network $(\S,\C,\Re)$, a \emph{(stochastic) kinetics} is an assignment of a function $\lambda_{y\to y'}:\Z_{\geq0}^d\to \R_{\geq0}$ to each reaction $y\to y'\in \Re$.  The time evolution of the copy numbers of species is then modeled by means of a continuous-time Markov chain with state space $\mathbb Z^d_{\ge 0}$,  whose transition rate from state $x$ to state $z$ is  given by
\begin{align*}
    q_{x,z}=\sum_{\substack{y\to y'\in\Re\\ y'-y=z-x}}\lambda_{y\to y'}(x),
\end{align*}
where the sum is over those reactions whose occurrence causes a net change that is precisely $z-x$.
If infinitely many transitions occur within a finite-time $T_\infty$ (in other words, an explosion occurs at $T_\infty$), we let 
  $X(t)=\Delta$ for any $t\geq T_\infty$, where $\Delta$ is a cemetery state not contained in $\Z_{\geq 0}^d$ \cite{NorrisMC97}.  The infinitesimal generator $\mathcal{A}$ of the associated Markov process acts on functions via the operation
\begin{align}\label{gen5}
 \mathcal{A}f(x)  =& \sum_{y \rightarrow y' \in \Re} \lambda_{y\rightarrow y'}(x)(f(x+y'-y)-f(x)),
\end{align}
for $f:\Z^d_{\ge 0} \to \R$   \cite{ethier2009markov}.

One of the most widely used choices of stochastic kinetics,
 and the one which the present paper focuses on exclusively, is given by  \emph{(stochastic) mass-action kinetics}, where for any reaction $y\to y'\in\Re$
\begin{align}\label{mass}
\lambda_{y\to y'}(x)=\kappa_{y\to y'} \frac{x!}{(x-y)!} \mathbbm{1}_{\{x \ge y\}} = \kappa_{y\to y'} \prod_{i=1}^d \frac{x_i!}{(x_i-y_i)!} \mathbbm{1}_{\{x_i \ge y_i\}},
\end{align}
for some positive constant $\kappa_{y\to y'}$, called a \emph{reaction rate constant} (or \emph{rate constant} for short).   We  denote $\K=\{\kappa_{y\to y'}\}$.  We may then denote the stochastic mass-action system by $(\S,\C,\Re,\K)$. 
To include the rate constants into the reaction graph, we place them right next to the arrow of the corresponding reaction as in
 $y\xrightarrow{\kappa_{y\to y'}} y'$.  Typical stochastic simulation algorithms for generating sample trajectories of this model are the Gillespie algorithm  \cite{Gill76,Gill77} or the next reaction method \cite{Anderson2007a,Gibson2000}, or are approximated via tau-leaping \cite{Anderson2007b,Gill2001}. Expectations and probabilities can be simulated by more advanced techniques \cite{AndHigham2012}.

\section{Main results: Exponential ergodicity of two classes of  reaction networks}\label{sec:mainresults}

In \cite{anderson2018some}, Anderson and Kim studied two classes of reaction networks and proved that any stochastic mass-action system associated with a network from one of these classes is  necessarily positive recurrent.  Moreover, they showed that positive recurrence holds regardless of the choice of reaction rate constants.  Hence, all such models  admit a stationary distribution. 
The different sets of conditions characterizing the classes are properties of the reaction graph  alone, and are easily checked.   In this paper, we prove that these models are also exponentially ergodic. Specifically, Theorem \ref{thm:single_mixing}  below states that networks satisfying the conditions detailed in Theorem 1 of \cite{anderson2018some} are exponentially ergodic.  Theorem \ref{thm:doublefull mixing}  below states that networks satisfying the conditions detailed in Theorem 2 of \cite{anderson2018some}  are exponentially ergodic and, moreover, that the convergence is uniform over the choice of initial condition.    After providing the results, with examples, we will close this section with two remarks that we hope shed light on the theorems.  We also provide generalizations to our results in Section \ref{sec:generalizations}.


Before stating the theorems, a few  definitions related to possible network structures are needed. These terms are introduced in \cite{anderson2018some}, but they are also universally used in this field.
\begin{defn}
A reaction network $(\S,\C,\Re)$ is called \emph{binary}  if $\|y\|_{\ell_1} = \sum_{i=1}^d y_i \le 2$ for all $y \in \C$. \hfill $\square$
\end{defn}

\begin{defn}
Let $(\S,\C,\Re)$ be a reaction network with $\S=\{ S_1,S_2,\cdots,S_d\}$.  
The complex $\emptyset$ is termed the \emph{zero complex}.  \emph{Unary complexes} and \emph{binary complexes} are the complexes of the form  $S_i$ and $S_i+S_j$, respectively, for $i,j\in\{1,\dots,d\}$, where we could have $i=j$. Binary complexes of the form $2S_i$ are called \emph{double complexes}.  If $2S_i\in \C$ for each $i=1,2,\dots,d$, then  the reaction network $(\S, \C, \Re)$ is said to be  \emph{double-full}. \hfill $\square$
\end{defn}

\begin{defn}
We will say that $(\S,\C,\Re)$ is \emph{open} if both $\emptyset \to S\in \Re$ and $S\to \emptyset \in \Re$ for each $S \in \S$. \hfill $\square$
\end{defn}

That is, the system is open if there are inflows and outflows for each species.  Note that in the case of open systems, all states communicate with each other and so $\Z^d_{\ge 0}$ is irreducible.

In the two papers \cite{AndGAC_ONE2011,AndBounded_ONE2011}, Anderson showed that \emph{deterministically} modeled weakly reversible reaction networks with a single linkage class are necessarily persistent and have bounded trajectories.  That is, they are ``well-behaved.''  Our first result pertains to a stochastic analog of those models.  That is, we start with a single linkage class that is weakly reversible and binary, and then add all inflow and outflow reactions (making the network ``open'').  Theorem 1 in \cite{anderson2018some} states that a stochastically modeled system associated with such a network is positive recurrent regardless of the choice of rate constants.  The result below gives us exponential ergodicity.

\begin{thm}\label{thm:single_mixing}
Let $(\S,\C, \Re)$ be a weakly reversible, binary reaction network that has a single linkage class and $\S=\{S_1,\dots,S_d\}$.   Let $\widetilde{\C} = \C \cup\{\emptyset\}\cup\{S_i \ | \ S_i \in \S\}$ and $\widetilde \Re =  \Re \cup_{S_i\in \S} \{\emptyset \to S_i, S_i \to \emptyset\}$. Then, for any choice of reaction rate constants $\K=\{\kappa_{y\to y'}\}$ there exist  $C,\eta>0$ so that the Markov process $X$ with intensity functions \eqref{mass} associated to the reaction network $(\S,\widetilde{\C},\widetilde \Re)$ satisfies
\begin{align}\label{eq:90878097}
\Vert P^t(x,\cdot) - \pi(\cdot) \Vert_{\TV} \le C(|x|+1) \ln(|x|+2) e^{-\eta t}, \quad \text{for all } t \ge 0 \text{ and all } x \in \Z^d_{\ge 0},
\end{align}
where $\pi$ is the unique stationary distribution of the model on $\Z^d_{\ge 0}$.
 Thus, the  mixing time  satisfies $\tau_x^\varepsilon = O(\ln(|x|))$, as $|x|\to \infty$.  
 
 Moreover, if there is a $w\in \R^{d}_{>0}$ for which $w\cdot (y'-y)=0$ for all $y\to y'\in \Re$, that is if the network $(\S,\C,\Re)$ is conservative, then 
 there exist  $C,\eta>0$ so that the Markov process $X$ with intensity functions \eqref{mass} associated to the reaction network $(\S,\widetilde{\C},\widetilde \Re)$ satisfies
\begin{align}\label{eq:90878098}
\Vert P^t(x,\cdot) - \pi(\cdot) \Vert_{\TV} \le C(|x|+1) e^{-\eta t}, \quad \text{for all } t \ge 0 \text{ and all } x \in \Z^d_{\ge 0}.
\end{align}
The mixing time of the model still satisfies $\tau_x^{\varepsilon} = O(\ln(|x|))$, as $|x|\to \infty$.
\end{thm}

We provide an example.

\begin{example}
\label{example:open-binary}
We take $\S = \{A,B,C\}$ and the binary reaction network with a single linkage class, $(\S,\C,\Re)$, to be
\begin{align*}
    &A \xrightarrow{\ \ \ \ \ \ } B.\\
&  \   \nwarrow \quad \swarrow   \\
& \hspace{0.60cm} 2C
\end{align*}
Adding inflows and outflows yields the reaction network $(\S,\widetilde C, \widetilde \Re)$, which has reaction graph
\begin{align}\label{eq:example_model1}
\begin{split}
    &\hspace{.33in}C \\
    &\hspace{.26in}\uparrow \downarrow\\
    &\hspace{.35in} \emptyset \\
    & \nearrow\hspace{-.07in}\swarrow \hspace{.2in} \nwarrow\hspace{-0.07in}\searrow\\
        &A \xrightarrow{\ \ \ \ \ \ \ \ \ } B\\
&  \   \nwarrow \hspace{.3in} \swarrow   \\
& \hspace{0.3in} 2C. 
\end{split}
\end{align}
Hence, by Theorem \ref{thm:single_mixing}, for any choice of rate constants the Markov model associated to the network $(\S,\widetilde \C, \widetilde \Re)$ is exponentially ergodic with mixing time bounded above by the logarithm of the initial counts. 

We provide the results of  a numerical example in Figure \ref{fig1}A.  Specifically, we consider the stochastic model associated with the network \eqref{eq:example_model1} in which all rate constants are selected to be equal to one.  In this case, the stationary distribution is a product of Poissons \cite{AndProdForm}. We then parameterized the model via the  initial condition $x_m = (m,m,m)^T$.  On the left hand side of Figure \ref{fig1}A we provide plots of $\|P^t(x,\cdot)- \pi(\cdot)\|_{\TV}$, as a function of time, for each of $m = 1, 10, 100$.  On the right hand side of Figure \ref{fig1}A we provide a plot of $\tau^\varepsilon_{x_m}$, for $x_m = (m,m,m)^T$ and $\varepsilon = 0.1$, as a function of $m$.  Note that this plot qualitatively agrees with Theorem \ref{thm:single_mixing} in that it appears logarithmic in $m$. The plots were made via approximating $P^t(x,\cdot)$ via Monte Carlo methods, and estimating the total variation distance via the approximation 
\[
    \Vert P^t(x,\cdot)-\pi(\cdot)\Vert_{\TV}=\frac{1}{2}\sum_{z}|P^t(x,z)-\pi(z)| \approx \frac{1}{2}\sum_{z \in [0,200]^3}|P^t(x,z)-\pi(z)|.
    \]
    \hfill $\triangle$
\end{example}

\begin{figure}[!hb]
    \centering
    \includegraphics[width=0.7\textwidth]{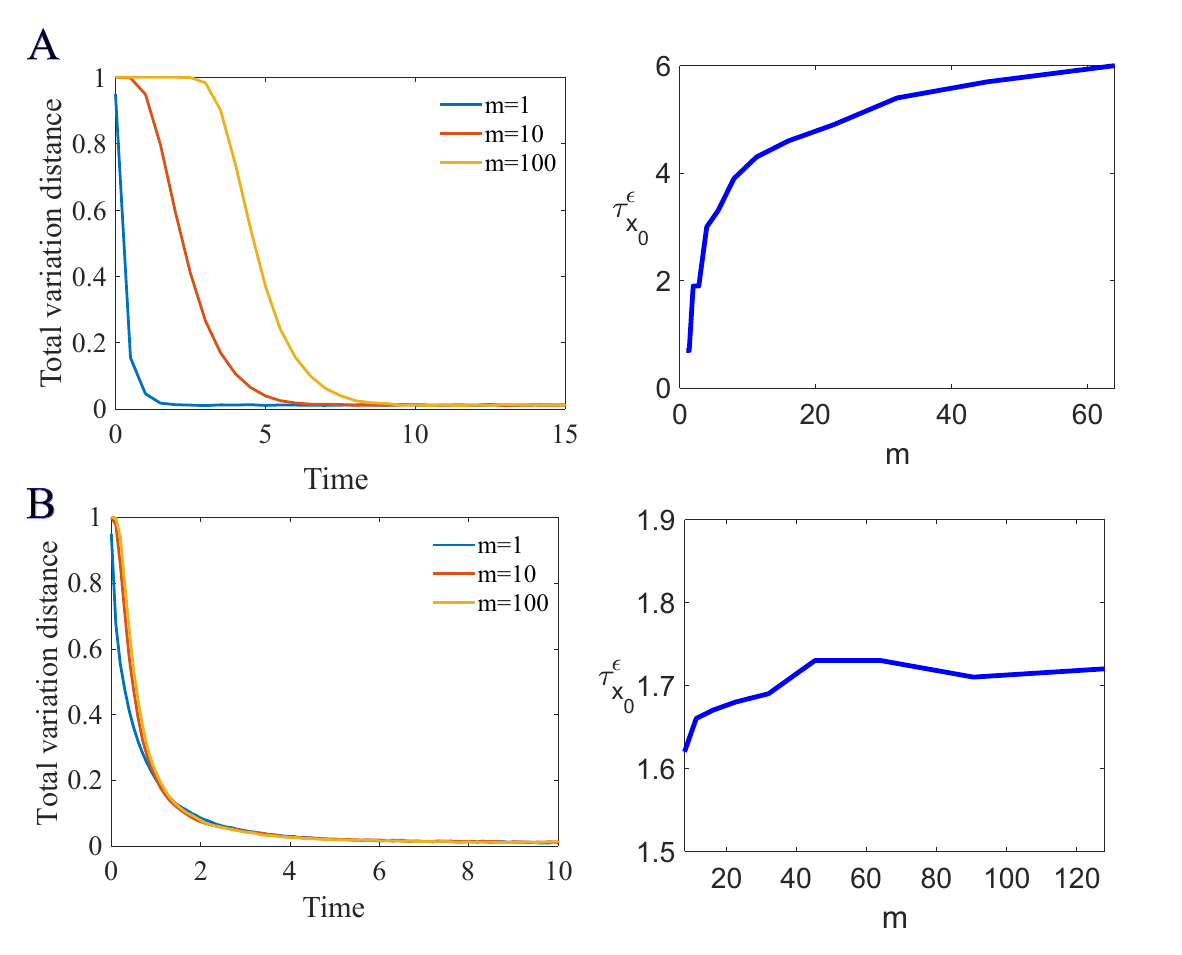}
    \caption{A and B (Left). For each initial condition $X(0)=(m,m,m)^\top$, the images on the left provide  plots of  $\Vert P^t(x,\cdot)-\pi(\cdot)\Vert_{\TV}$ over time for the reaction networks \eqref{eq:example_model1} (top) and \eqref{eq:ex of double} (bottom), where $\pi$ represents their stationary distributions. \\
    A and B (Right). The images on the right provide plots  of the mixing times $\tau^\varepsilon_{x_m}$ with $\varepsilon=0.1$ and $x_m=(m,m,m)^\top$, as functions of $m$.  The top image is for the reaction network \eqref{eq:example_model1} whereas the bottom image is for the the reaction network \eqref{eq:ex of double}.   Note that the top image   appears to show logarithmic growth whereas the bottom appears to show an asymptotic upper bound.  Both observations agree with our theory.
    }
    \label{fig1}
\end{figure}



The next theorem, which allows us to conclude that the convergence is also uniform over initial states,  drops both the weak reversibility and the single linkage class assumption.  However, it adds the assumption that the network is double-full and an assumption related to  paths through the reaction graph that start with double complexes. Another difference with the previous theorem, though not the corollay, is that now $\Z^d_{\ge 0}$ may no longer be irreducible. This is also an extension of Theorem 2 in \cite{anderson2018some} that only shows ergodicity of the proposed reaction network class.

\begin{thm} \label{thm:doublefull mixing}
Let $(\S,\C, \Re)$ be a binary reaction network satisfying the following two conditions:
\begin{enumerate}[(i)]
\item the reaction network is double-full, and
\item for each double complex (of the form $2S_i$) there is a directed path within the reaction graph beginning with the double complex itself and ending with either a unary complex (of the form $S_j$) or  the zero complex, $\emptyset$.
\end{enumerate} 
Let  $\K=\{\kappa_{y\to y'}\}$ be a choice of reaction rate constants.
Let $X(0) = x \in \mathbb{S} \subset \Z^d_{\ge 0}$ be in a closed communication class, $\mathbb{S}$,  and let $\pi$ be the stationary distribution of the model on that class.  Then,   there exist  $C,\eta>0$ so that the Markov process $X$ with intensity functions \eqref{mass} associated to the reaction network $(\S,\C, \Re)$ satisfies
\begin{align*}
\Vert P^t(x,\cdot) - \pi(\cdot) \Vert_{\TV} \le C e^{-\eta t}, \quad \text{for all } t \ge 0,
\end{align*}
where neither $C$ nor $\eta$ depend upon $x \in \mathbb{S}$.
 Thus, the mixing time   satisfies $\tau_x^\varepsilon = O(1)$, as $|x|\to \infty$.
\end{thm}

We provide an example.

\begin{example}\label{example:network_type2}
Consider the following network
\begin{align}\label{eq:ex of double}
\begin{split}
& 2A \rightleftarrows   A+B \rightleftarrows B, \quad  A\rightleftarrows  2C \rightleftarrows B+C.  \\
&2B \rightleftarrows \emptyset,  \quad C \rightleftarrows A+C
\end{split}
\end{align}
This network is double full since $2A, 2B, 2C \in \C$.  Moreover, for each double complex, there is a sequence of directed edges starting with the double complex and arriving at a complex that is either $\emptyset$ or a unary complex.  Specifically, we have the following such sequences:
\begin{align*}
    &2A \to A + B \to B\\
    &2C \to A\\
    &2B \to \emptyset
\end{align*}
Hence, by Theorem \ref{thm:doublefull mixing}, for any choice of rate constants the associated Markov model is exponentially ergodic.  Moreover, the convergence time is uniform over the initial state. 

As in Example \ref{example:open-binary}, we 
 provide the results of  a numerical example in Figure \ref{fig1}B.  Specifically, we consider the stochastic model associated with the network \eqref{eq:ex of double} in which all rate constants are selected to be equal to one.  In this case, the stationary distribution is a product of Poissons \cite{AndProdForm}. We then parameterized the model via the  initial condition $x_m = (m,m,m)^T$.  On the left hand side of Figure \ref{fig1}B we provide plots of $\|P^t(x,\cdot)- \pi(\cdot)\|_{\TV}$, as a function of time, for each of $m = 1, 10, 100$.  On the right hand side of Figure \ref{fig1}B we provide a plot of $\tau^\varepsilon_{x_m}$, for $x_m = (m,m,m)^T$ and $\varepsilon = 0.1$, as a function of $m$.  Note that this plot qualitatively agrees with Theorem \ref{thm:doublefull mixing} in that it appears to be bounded above in $m$. The plots were made via approximating $P^t(x,\cdot)$ via Monte Carlo methods, and estimating the total variation distance via the approximation 
\[
    \Vert P^t(x,\cdot)-\pi(\cdot)\Vert_{\TV}=\frac{1}{2}\sum_{z}|P^t(x,z)-\pi(z)| \approx \frac{1}{2}\sum_{z \in [0,200]^3}|P^t(x,z)-\pi(z)|.
    \]
\hfill $\triangle$
\end{example}

We close this section with two remarks.

\begin{rem}
We note that the exponential convergence bounds in Theorem \ref{thm:single_mixing} (given in \eqref{eq:90878097} and \eqref{eq:90878098}) are not uniform over the initial state, $x$.  This can be understood by the fact that the ``strength'' of the convergence is governed by the linear outflow reactions $S_i \to \emptyset$.  Such linear terms lead to exponential decay and, hence, logarithmic mixing times.  Conversely, for Theorem \ref{thm:doublefull mixing} the decay has quadratic rate, which should be compared to the ODE $\dot x(t) = -x(t)^2$, which implodes from infinity (i.e., the time needed to hit a given compact from a particular state   is bounded from above).
\end{rem}

\begin{rem}
To prove Theorems \ref{thm:single_mixing} and \ref{thm:doublefull mixing} we require a number of  results that we collect in the next section.  However, we will provide a bit of intuition here.  The main point is that we must show that the ``average drift'' of the model (which will be clarified in the next section through the use of Foster-Lyapunov functions) points inward toward the origin from \textit{each} $x \in \Z^d_{\ge 0}$ outside of a compact set.  There are then two main cases to consider: (i) those $x$ that are ``away'' from the boundary in that there are sufficient counts for each reaction to have positive intensity, and (ii) those $x$ near the boundary in that at least one of the reactions has zero intensity (for example, perhaps $x_i = 0$ for some $i$).  The case of $x$ ``away'' from the boundary can be handled in a manner that is similar to the deterministic case \cite{AndBounded_ONE2011, AndGAC_ONE2011}, and relies on the overall structure of the networks.  The existence of the outflow reactions in Theorem \ref{thm:single_mixing} and of the double complex condition in Theorem \ref{thm:doublefull mixing} ensure that the  drift towards the origin also takes place for the $x$ near the boundary. 
\end{rem}


\section{Preliminary concepts: tiers and Foster-Lyapunov functions}
\label{sec:preliminary}

This brief section is dedicated to introducing two key analytical tools we will use in the proofs of Theorems \ref{thm:single_mixing} and \ref{thm:doublefull mixing}: tiers and Foster-Lyapunov functions.   

\subsection{Tiers}\label{subsec:tiers}
``Tiers'' were introduced to allow for the partition of the complexes along sequences in a manner that tells us which reactions/complexes are ``dominating'' the dynamics of the mathematical model in different regions of the state space.  Tiers were first introduced by Anderson in \cite{AndGAC_ONE2011, AndBounded_ONE2011} in the context of deterministically modeled reaction networks, but have recently been successfully used to analyze stochastic models as well \cite{anderson2018some,anderson2020tier}.  


There are multiple ways to partition the set of complexes along a particular sequence of points.  Below, and for reasons that will become clear throughout our analysis, we choose to do so via the function $x\mapsto (x \vee 1)$, which is the vector whose $j$th component is
\[
(x\vee 1)_j = x_j\vee 1 = \max\{x_j,1\}.
\]
This particular choice of function (and hence partition) was first used in \cite{anderson2018some} where it was called a ``D-type partition.''

\begin{defn}\label{def:tier-seq}
Let $(\S,\C,\Re)$ be a reaction network with  $\S=\{S_1,\dots,S_d\}$.   A sequence $\{x_n\}$ of points in $\R^d_{\ge 0}$ is called a \emph{D-type tier-sequence}  if
\begin{enumerate}[(i)]
\item  for each $i\in \{1,\dots,d\}$ the limit $\displaystyle \lim_{n\rightarrow\infty}x_{n,i}$ exists (it could be infinity)  and  $\displaystyle \lim_{n\rightarrow \infty}x_{n,i} = \infty$ for at least one $i$, and
\item for any pair of complexes $y,y'\in \C$, the  limit 
\begin{align*}
   \lim_{n\rightarrow \infty} \frac{(x_n \vee 1)^{y}}{(x_n\vee 1)^{y'}},
\end{align*}  
exists (it could be infinity).  \hfill $\square$
\end{enumerate}
\end{defn}
  
\begin{rem}\label{rem:tier seq}
Note that, given a sequence $\{x_n\}$ of vectors in $\R^d_{\ge0}$ with $\limsup_{n\to \infty} \|x_n\|_\infty = \infty$, it is always possible to find a subsequence that is a D-type tier-sequence.  This is simply because we assume that only finitely many complexes are present.  See Lemma 4.2 of \cite{AndGAC_ONE2011}.
\end{rem}

\begin{defn}\label{def31}
Let $( \S,\C,\Re)$ be a  reaction network and let $\{x_n\}$ be a D-type tier-sequence in $\R^d_{\ge 0}$. We define a partition of the complexes $\C =  \cup_{i=1}^K T^{D,i}_{\{x_n\}}$  along the tier-sequence $\{x_n\}$ in the following recursive manner:
\begin{enumerate}[(i)]
\item  we say that a complex $y$ is in \emph{tier 1} (and write $y \in T^{D,1}_{\{x_n\}}$) if for all complexes $y'\in \C$
\begin{align*}
\lim_{n\rightarrow \infty} \frac{(x_n \vee 1)^{y}}{(x_n\vee 1)^{y'}} >0;
\end{align*}
\item we say that a complex $y$ is in \textit{tier} $i$ (and write $y \in T^{D,i}_{\{x_n\}}$) if there exists $y'\in T^{D,i-1}_{\{x_n\}}$ with
\[
    \lim_{n\to \infty}  \frac{(x_n \vee 1)^{y}}{(x_n\vee 1)^{y'}} =0
\]
and for all complexes $y'\notin \bigcup_{j=1}^{i-1} T^{D,j}_{\{x_n\}}$ we have
\[
    \lim_{n\to \infty}  \frac{(x_n \vee 1)^{y}}{(x_n\vee 1)^{y'}} >0.
\]
\end{enumerate}
The mutually disjoint subsets $T^{D,i}_{\{x_n\}}$ are called \textit{D-type tiers along $\{x_n\}$}.  If $y\in T^{D,i}_{\{x_n\}}$ and $y' \in T^{D,j}_{\{x_n\}}$ with $i<j$ we will write $y \succ_{D} y'$.  If $y,y'$ are in the same D-type tier, then we will write $y \sim_{D} y'$.
\hfill $\square$
\end{defn} 

Thus, those complexes in tier 1 maximize $(x_n\vee 1)^y$ along the sequence, with those in tier 2 being second largest, etc.  Note that as a consequence of the above definition, 
if  $y,y' \in T^i_{\{x_n\}}$ then there exists $C\in (0,\infty)$ such that
\begin{align*}
\lim_{n\rightarrow \infty} \frac{(x_n \vee 1)^{y}}{(x_n\vee 1)^{y'}} =C, 
\end{align*}
explaining the notation $y \sim_D y'$, and  if $y\in T^i_{\{x_n\}}$ and $y'\in T^{j}_{\{x_n\}}$ for some $i<j$, then 
\begin{align*}
\lim_{n\rightarrow \infty} \dfrac{(x_n \vee 1)^{y'}}{(x_n\vee 1)^{y}} = 0,
\end{align*}
explaining the notation $y\succ_D y'$.

The following example demonstrates the concept.

\begin{example}\label{ex:tier}
Consider the reaction graph below.
\begin{align*}
    &2A \ \ \longrightarrow \ \ B \rightleftharpoons \emptyset \rightleftharpoons A. \\[-0.5ex]
    &  \ \displaystyle \nwarrow \hspace{0.8cm}  \swarrow \\[-0.5ex]
& \hspace{0.4cm} A+B
\end{align*}
Let $\C$ be enumerated with $y_1=2A, y_2=B, y_3=\emptyset, y_4=A, y_5=A+B$.
For $x_n=(n,0)^\top$, $(x_n\vee 1)^{y_i}$ is equal to $n^2, 1, 1, n$ and $n$ for $i=1,2,3,4$ and $5$, respectively. 
On the other hand, if $x_n=(n,n+1)^\top$, then $(x_n\vee 1)^{y_i}$ is equal to $n^2, n+1, 1, n$ and $n(n+1)$ for $i=1,2,3,4$ and $5$, respectively.
Note that both sequences are  tier-sequences.
The tier structures (i.e., the partition) for each tier-sequence is  shown in the following table.

\begin{center}
\begin{tabular}{ M{2cm}| M{2cm}| P{2cm}}
     & & \\[-2ex]
     $x_n$& $(n,0)^\top $ & $(n,n+1)^\top$  \\[1ex]
     \hline 
     & & \\[-2ex]
     $T^{D,1}_{\{x_n\}}$ & $y_1$ & $y_1,y_5$\\[1ex]
     \hline
     & & \\[-2ex]
     $T^{D,2}_{\{x_n\}}$ & $y_4, y_5$ & $y_2, y_4$\\[1ex]
     \hline
     & &  \\[-2ex]
     $T^{D,3}_{\{x_n\}}$ & $y_2, y_3$ & $y_3$\\[1ex]
\end{tabular}
\end{center}

\end{example}

 We close this section with a  useful lemma.

 \begin{lem}\label{lem22}  Let $\{x_n\}$, with $x_n \in \Z^d_{\ge 0}$,   be a D-type tier-sequence of a reaction network $(\S,\C,\Re)$ and let $y \to y'\in \Re$. If $\lim_{n\to \infty} \lambda_{y\to y'}(x_n)>0$, where $\lambda_{y\to y'}$ is given in \eqref{mass}, then there is a $C>0$ with \begin{align*}
&\lim_{n \rightarrow \infty} \dfrac{\lambda_{y\to y'}(x_n)}{(x_n \vee 1)^{y}} = C.
\end{align*}
Moreover, if  the source complex $y$ satisfies $|y_i|\le 1$ for each $i \in \{1,\dots, d\}$, then $C = \kappa_{y\to y'}$.
\end{lem}

\begin{proof}
Since $\lim_{n\to \infty} \lambda_{y\to y'}(x_n)>0$, we have that $(x_n\vee 1)^y = x_n^y$ for all $n$ large enough.  The result then follows from the basic polynomial structure of $\lambda_{y\to y'}$.
\end{proof}

\subsection{Lyapunov functions for mixing times}\label{app:lya}

The following known result is key to our analysis.  See, for example, \cite{MT-LyaFosterIII, Tweedie1981}.

\begin{thm}\label{thm:exp_ergo}
Let $X$ be a continuous-time Markov chain on an irreducible,  countably infinite state space $\mathbb{S} \subset \Z^d$ with generator $\mathcal A$. Suppose there exists a positive function $V$ on $\mathbb{S}$ satisfying the following. 
\begin{enumerate}[(i)]
\item $V(x) \to \infty$, as $|x|\to \infty$, and
\item there are positive constants $a$ and $b$ such that
\begin{eqnarray}\label{eq:exp_ergo}
\mathcal{A}V(x) \le -aV(x)+b \quad \text{for all} \ \ x \in \mathbb{S}.
\end{eqnarray}
\end{enumerate}
Then   there exists positive constants $\eta$ and $C$ such that
\begin{align*}
\Vert P^t(x,\cdot) - \pi(\cdot) \Vert_{\TV} \le C(V(x)+1)e^{-\eta t}, \quad \forall x \in    \mathbb{S}.
\end{align*}
\end{thm}

Note that if a continuous-time Markov chain $X$ satisfies the conditions in \thmref{thm:exp_ergo} with the   function $V$ defined from \eqref{eq:StandardLyapunov}, 
then the mixing time satisfies $\tau_x^\varepsilon \le \frac{1}{\eta}{\ln{\left(\frac{C(V(x)+1)}{\varepsilon}\right)}} = O(\ln|x|)$. 

The next result utilizes a `super Lyapunov function.'  Such functions will allow us to show that, for certain models,  the convergence is uniform over the initial state.   
This theorem is a version of Theorem 3.2 in \cite{superlya2012}. 

\begin{thm}\label{thm:superlya}
Let $X$ be a continuous-time Markov chain on an irreducible, countable state space $\mathbb{S}\subset \R^d$ with generator $\mathcal A$. Suppose there exists a positive function $V$ on $\mathbb{S}$ satisfying the following. 
\begin{enumerate}[(i.)]
\item $V(x) \to \infty$, as $|x|\to \infty$, and
\item there are positive constants $a$, $b$ and $\delta$ such that
\begin{eqnarray}\label{eq:uni_ergo}
\mathcal{A}V(x) \le -aV(x)^{1+\delta}+b \quad \text{for all} \ \ x \in \mathbb{S}.
\end{eqnarray}
\end{enumerate}
Then there exist positive constants $\eta$ and $C$ such that 
\begin{align*}
\Vert P^t(x,\cdot) - \pi(\cdot) \Vert_{\TV} \le C e^{-\eta t}.
\end{align*}
\end{thm}
Note that in this case we have that the mixing time satisfies $\tau_x^\varepsilon = O(1)$.

\section{Proofs of Theorems \ref{thm:single_mixing} and \ref{thm:doublefull mixing}}
\label{sec:proofs}

 We let
$V:\Z^d_{\ge 0} \to \R_{\ge 0}$ be the following function
\begin{align}\label{eq:StandardLyapunov}
    V(x) = \sum_{i=1}^d \left[ x_i (\ln(x_i) - 1) + 1\right],
\end{align}
with $0(\ln(0)-1)$ taken to be zero.
This function has a long history in the field of chemical reaction network theory:  it has been used to show stability of deterministically modeled reaction networks \cite{Feinberg72, AndBounded_ONE2011, Feinberg87, HornJack72}, to show positive recurrence of stochastic reaction networks \cite{anderson2018some, anderson2020tier, anderson2020stochastically}, and also plays a key role in the study of ``non-equilibrium'' systems in the physics literature (see \cite{anderson2015lyapunov} and followup references).
The basic idea behind the proofs of Theorems \ref{thm:single_mixing} and \ref{thm:doublefull mixing} is to use the infinitesimal behavior of the  Markov process, which can be described via the generator $\A$ defined in \eqref{gen5}, and the  Foster-Lyapunov function $V$ defined above in   \eqref{eq:StandardLyapunov},
to show that the models under consideration admit an ``inward'' drift at every state outside some compact set, in the sense of Theorems \ref{thm:exp_ergo} and \ref{thm:superlya}.  To do so we make use of the tier structures of the model.

We begin by first providing a series of lemmas.


\begin{lem}\label{lemma:tier_lya}
Let $\mathcal A$ be the generator \eqref{gen5} of the continuous-time Markov chain associated to a reaction network $(\S,\C,\Re)$ with mass-action kinetics \eqref{mass} and rate constants $\{\kappa_{y\to y'}\}$. Let $V$ be the function defined in \eqref{eq:StandardLyapunov}.  For each D-type tier-sequence $\{x_n\}$ 
there is a constant $C_1>0$  for which
\begin{align}\label{asymptotic of AV} \A V(x_n) \le \sum_{y\rightarrow y'\in \Re} \lambda_{y\to y'}(x_n)\Big ( \ln{\left(\frac{(x_n \vee 1)^{y'}}{(x_n \vee 1)^y}\right) +C_1 \Big )}.
\end{align}
Furthermore, suppose that there exists a reaction $y_1\to y'_1\in \Re$ such that 
\begin{enumerate}[(i)]
\item $y_1 \in T^{D,1}_{\{x_n\}}$, 
\item $y_1 \succ y_1'$, and 
\item  $\dlim_{n\to \infty} \dfrac{\lambda_{y_1\to y_1'}(x_n)}{(x_n\vee 1)^{y_1}} = C_2$, for some $C_2>0$.
\end{enumerate}
Then there exists a positive constant $C_3$ such that for large enough $n$,
\begin{align}\label{asymptotic of AV 2}
    \mathcal AV(x_n) \le C_3\sum_{\substack{y\to y'\in \Re\\
    y \succ_D y'} }
    \lambda_{y\to y'}(x_n) \ln{\left(\frac{(x_n\vee 1)^{y'}}{(x_n\vee 1)^{y}}\right)},
\end{align}
where the sum is over those reactions, $y\to y'\in \Re$, with $y \succ_D y'$.
 \end{lem}

The proof of \eqref{asymptotic of AV} can be found in \cite[Lemma 8]{anderson2018some}. The proof of \eqref{asymptotic of AV 2} can be found in the proof of \cite[Theorem 9]{anderson2018some}.

The next lemma gives a condition for when \eqref{asymptotic of AV 2} holds.

\begin{lem}\label{lemma:tier_topology}
Let $(\S,\C, \Re)$, with $\S=\{S_1,\dots,S_d\}$, be a weakly reversible, binary reaction network that has a single connected component. Suppose that  $\C$ is not a subset of $\{S_i+S_j : i,j \in \{1,\dots, d\}\}$. Let $\widetilde{\C} = \C \cup\{\emptyset\}\cup\{S_i \ | \ S_i \in \S\}$ and $\widetilde \Re =  \Re \cup_{S_i \in \S} \{\emptyset \to S_i, S_i \to \emptyset\}$.
 Then, for any D-type tier-sequence $\{x_n\}$ for $(\S,\widetilde \C,\widetilde \Re)$, there exists a reaction $y_1 \to y_1' \in \widetilde \Re$ such that
 \begin{enumerate}[(i)]
 \item $y_1 \in T^{D,1}_{\{x_n\}}$, 
 \item $y_1 \succ_{D} y_1'$ and 
 \item $\dlim_{n\to \infty} \dfrac{\lambda_{y_1\to y_1'}(x_n)}{(x_n\vee 1)^{y_1}} = C$, for some $C>0$. 
 \end{enumerate}
\end{lem}

\begin{proof}
First note  that $\emptyset \notin \Td1$.  This follows because there is at least one $i\in \{1,\dots, d\}$ with  $x_{n,i}\to \infty$, since $\{x_n\}$ is a D-type tier sequence,   and so  $S_i \succ_D \emptyset$. 

Next, we observe that if  $S_i \in \Td1$ for some $i$, then we are done because in this case the reaction $S_i \to \emptyset$ satisfies all of $(i), (ii),$ and $(iii)$. 

Now suppose that there is no $S_i \in \Td1$. Then it must be the case that $S_i +S_j \in \Td1$ for some $i,j$ (we could have $i = j$).  However, because $\C$ does not consist solely of binary complexes, we must have at least one complex, $\tilde y\in \C$, with $|\tilde y| \in \{0,1\}$ (i.e., it is either unary or the zero complex).  In particular, we know that $\tilde y \notin \Td1$.
 Hence,  by the weak reversibility of $(\S,\C,\Re)$, there is a sequence of reactions starting with $S_i+S_j$ and ending with a complex not contained in $\Td1$.  This implies that there is at least one reaction, $y\to y'\in \Re$, with $y \in \Td1$ and $y \succ_D y'$.  That is, conditions (i) and (ii) are satisfied for this reactions.
 
 We now look closer at $y$, the source complex of this reaction.  If $y = 2S_k$ for some $k$, then we have that this reaction also satisfies (iii) with $x_{n,k} \to \infty$, as $n\to \infty$, and we are done.  Similarly, if $y = S_k + S_\ell$ for some $k\ne \ell$, and both $\lim_{n\to \infty} x_{n,k} > 0$ and $\lim_{n\to \infty} x_{n,\ell} > 0$, then   Lemma \ref{lem22} may be used to imply that (iii) holds as well, and we are done.  The final case to consider is when $y= S_k + S_\ell$ with $\lim_{n\to \infty} x_{n,k} > 0$ and $\lim_{n\to \infty} x_{n,\ell} = 0$.  However, this would imply that $S_k \in \Td1$ as well, which is impossible by assumption.  Hence, all cases are covered and the proof is complete.
\end{proof}

While the inflow reactions are not crucial, the outflows in $\widetilde{\mathcal R}$ is necessary for Lemma \ref{lemma:tier_topology}.  The next example demonstrates.

\begin{example}
Let $(\S,\C,\Re)$ be a reaction network formed with the reactions in the `triangle' within the reaction network shown in Example \ref{ex:tier}.  That is, 
\begin{align*}
    &2A \ \ \longrightarrow \ \ B  \\[-0.5ex]
    &  \ \displaystyle \nwarrow \hspace{0.8cm}  \swarrow \\[-0.5ex]
& \hspace{0.4cm} A+B.
\end{align*}
Let $(\S,\widetilde \C,\widetilde \Re)$ be the full network in Example \ref{ex:tier}, where all the inflows and outflows are included. For the D-type tier-sequence $x_n=(n,0)^\top$,  the reaction $2A\to B$ satisfies all the conditions (i), (ii) and (iii) in Lemma \ref{lemma:tier_topology}. For the D-type tier sequence $x_n=(0,n)^\top$, the  reaction $A+B \to 2A$ satisfies (i) and (ii) in Lemma \ref{lemma:tier_topology}. However (iii) does not hold because the rate function $\lambda_{A+B\to 2A}$ is zero at $x_n$ for each $n$. In this case, the out-flow reaction $B\to \emptyset$ satisfies (i), (ii), and (iii) in Lemma \ref{lemma:tier_topology}. \hfill $\triangle$
\end{example}

Now we begin the proof of Theorem \ref{thm:single_mixing}. 




\begin{proof}[Proof of Theorem \ref{thm:single_mixing}]
We first suppose the existence of a $w\in \R^d_{>0}$ with $w\cdot (y'-y)=0$ for all $y \to y'\in \Re$. Let $W(x) = w\cdot x = \sum_{i = 1}^d w_i x_i$.  Note that  $\displaystyle \lim_{|x|\to \infty} |W(x)| = \infty$ since $w \in  \mathbb R^d_{>0}$.  For each $i \in \{1,\dots, d\}$, let $e_i\in \Z^d_{\ge 0}$ be the vector with a 1 in the $i$th coordinate and zeros elsewhere.  Since $w$ conserves the reactions in $\Re$, we have
\begin{align*}
\mathcal{A}W(x) 
&= \sum_{y\to y'\in \Re} \lambda_{y\to y'}(x) (W(x+y'-y) - W(x)) \\
&\hspace{.1in}+ \sum_{i=1}^d  \lambda_{S_i \to \emptyset } (x) (W(x-e_i)-W(x))+ \sum_{i=1}^d  \lambda_{\emptyset \to S_i} (x) (W(x+e_i)-W(x))\\
&= - \sum_{i=1}^d \kappa_{S_i \to \emptyset} \cdot w_i x_{i} + \sum_{i=1}^d \kappa_{\emptyset \rightarrow S_i} w_i \\
&\le -aW(x) + b,
\end{align*}
where $a = \min\{\kappa_{S_i \to \emptyset} \}$ and $b = \sum_{i=1}^d \kappa_{\emptyset \rightarrow S_i} w_i$.  
Hence, \eqref{eq:exp_ergo} holds with $V$ replaced with $W$, giving us the second statement of the theorem, by Theorem \ref{thm:exp_ergo}.

 Now we assume that no such $w$ exists.  Note that this implies  that neither $\C \subset \{2S_i : S_i \in \S\}$ nor $\C \subset \{S_i : S_i \in \S\}$.  We claim that there exists a compact set $K \subset \mathbb{S}$ such that for some positive constant $a$, the function $V$ defined in \eqref{eq:StandardLyapunov} satisfies
\begin{align}\label{eq:claim for single expo}
\mathcal{A}V(x) \le -aV(x) \quad \text{for all} \ \ x\in \mathbb{S}\setminus K.
\end{align}
If this claim holds, then the result follows from Theorem \ref{thm:exp_ergo} after setting $b=(a+1)\displaystyle \max_{x\in K} \mathcal AV(x)$.

We prove the claim by contradiction. We therefore suppose that there is not a compact $K$, together with an $a>0$, so that \eqref{eq:claim for single expo}  holds. In this case,  there exists a sequence $\{x_n\}\subset\mathbb{S}$ such that
\begin{align}\label{eq:negation for single expo}
\lmt||x_n||_{\infty}=\infty \quad \text{and} \quad \mathcal AV(x_n)\ge -\frac{1}{n}V(x_n), \ \ \text{for all} \ \ n.
\end{align}
By Remark \ref{rem:tier seq}, we can find a subsequence of $\{x_n\}$ that is a D-type tier-sequence. For ease of notation, we also denote this subsequence by $\{x_n\}$. By considering a further subsequence if needed, we also can assume that there exists an $i$ such that
\begin{align}\label{eq:88787878}
    \lim_{n\to \infty}\frac{x_{n,i}}{x_{n,j}} >0, \quad \text{for any $j$,}
\end{align}
where we denote the $i$th  coordinate of $x_n$ by $x_{n,i}$.  Because of \eqref{eq:88787878}, we will say that  $x_{n,i}$ is  a ``maximal coordinate,'' since no other element is asymptotically larger. Without loss of generality, we assume $x_{n,1}$ is such a maximal coordinate of $x_n$. That is, there exists a positive constant $C'$ such that for each $i$
\begin{align*}
x_{n,i} \le C'x_{n,1}, \quad \text{for all $n$ large enough}.
\end{align*}
 
By Lemmas \ref{lemma:tier_lya} and \ref{lemma:tier_topology} there is a $C_1>0$ for which  \eqref{asymptotic of AV 2} holds for all $n$ large enough. That is,
\begin{align}\label{asymptotic of AV 3}
\mathcal AV(x_n) \le C_1 \sum_{\substack{y\to y'\in \Re\\
y \succ_D y'} } \lambda_{y\to y'}(x_n) \ln{\left(\frac{(x_n\vee 1)^{y'}}{(x_n\vee 1)^{y}}\right)}.
\end{align}
Note that because we are assuming that $x_{n,1}$ is a maximal coordinate,  it holds that $S_1 \succ \emptyset$, and so the reaction $S_1 \to \emptyset$ is included in the sum above.  Since each term in the summation of the right-hand side of \eqref{asymptotic of AV 3} is negative for large $n$, the inequality stands after dropping other terms except for the term related to $S_1 \to \emptyset$. Hence we have
\begin{align*}
\A V(x_n) \le C_1 \lambda_{S_1 \to \emptyset}(x_n)\ln{\left(\frac{(x_n\vee  1)^{(0,0,\dots,0)}}{(x_n\vee 1)^{(1,0,\dots,0)}}\right)}= -C_1\lambda_{S_1\to \emptyset}(x_n)\ln{(x_{n,1}\vee 1)}.
\end{align*}
However, since we assumed $x_{n,1}$ is a maximal coordinate of $x_n$, and $\lmt x_{n,1}=\infty$ by construction,   we can find some positive constant $C_2$ such that $V(x_n) \le C_2 x_{n,1}\ln(x_{n,1})$ for all $n$ large enough. Combining the above, there is a $C_3>0$ so that for all $n$ large enough we have
\begin{align*}
\A V(x_n) &\le -C_1\lambda_{S_1 \to \emptyset}(x_n)\ln{(x_{n,1}\vee 1)}.\\
& \le -C_3 V(x_n).
\end{align*}
This contradicts \eqref{eq:negation for single expo}, and the result is shown.
\end{proof}
\begin{rem}
 Since $(\S,\widetilde \C, \widetilde \Re)$ is weakly reversible in Theorem \ref{thm:single_mixing}, the associated Markov process $X$ with $X(0)=x$ is irreducible.  Moreover, because $\widetilde \Re$ is open, the irreducible state space is all of $\Z^d_{\ge 0}$. 
\end{rem}

We now turn to the proof of Theorem \ref{thm:doublefull mixing}.
We begin with a lemma that characterizes tiers of double-full binary reaction networks.

\begin{lem}\label{lem:tier and double}
Let $(\S,\C, \Re)$ be a binary reaction network satisfying the same conditions given in Theorem \ref{thm:doublefull mixing}.
Let $\{x_n\}$ be a D-type tier-sequence.  Then the following hold. 
\begin{enumerate}[(i)]
    \item For $y\to y'\in \Re$ with $y \in T^{D,1}_{\{x_n\}}$, we have that $\lambda_{y\to y'}(x_n) > 0$ for all  $n$ large enough,
    \item $2S_i \in T^{D,1}_{\{x_n\}}$ for some $i$, and
    \item There exists $y\to y' \in \Re$ such that $y\in T^{D,1}_{\{x_n\}}$ and $y \succ y'$.
\end{enumerate}
\end{lem}

\begin{proof}
There are no complexes with $\|y\|_{\ell_1} > 2,$ and our assumptions imply that each double complex is the source complex for some reaction. Hence,  (i) and (ii)  follow.  

To see that (iii) holds, you simply have to note that no unary complex (or the zero complex) can be in $T^{D,1}_{\{x_n\}}$.  (iii) then follows from  the existence of a directed path from any double complex to a unary complex or the zero complex. 
\end{proof}

We now prove Theorem \ref{thm:doublefull mixing}. 


\begin{proof}[Proof of Theorem \ref{thm:doublefull mixing}]
We will show that $V$ defined in \eqref{eq:StandardLyapunov} satisfies the conditions of Theorem \ref{thm:superlya} with $\delta = 1/2$ (and value $\delta \in (0,1)$ would work). 
As in the proof of Theorem \ref{thm:single_mixing}, we proceed via a proof by contradiction. Therefore, in order to find a contradiction, we assume that there is no compact set $K\subset \mathbb{S}$, together with an $a>0$, for which  
\begin{align*}
    \mathcal AV(x)\le -a V(x)^{3/2} \quad \text{for all $x \in \mathbb S\setminus K$}.
\end{align*}
Then there exists a sequence $\{x_n\}$, with $x_n \in \mathbb{S}$, so that
\begin{align}\label{eq:negation for double}
\lmt \|x_n\|_{\infty} =\infty \quad \text{and} \quad \mathcal AV(x_n)\ge -\frac{1}{n}V(x_n)^{3/2}, 
\end{align}
for all $n$.  By Remark   \ref{rem:tier seq}, we may extract a subsequence that is a D-type tier sequence and that still satisfies \eqref{eq:negation for double}.  As before, for ease of notation we denote this subsequence by $\{x_n\}$.

By result (ii) in Lemma \ref{lem:tier and double}, there is a species $S_i$ for which  $2S_i \in T^{D,1}_{\{x_n\}}$. We will now show that  there must be a $C>0$ with 
\[
    \mathcal AV(x_n) \le -C x_{n,i}^2,
    \]
    which would contradict \eqref{eq:negation for double}.

By    Lemma \ref{lem:tier and double}, there is a $y_1\to y_1'\in \Re$ for which  $y_1 \in T^{D,1}_{\{x_n\}}$ and $y_1 \succ y'_1$.  By Lemma \ref{lem22} and Lemma \ref{lemma:tier_lya} (i) we also have  that the limit in (iii) of  Lemma \ref{lemma:tier_lya} holds.  Hence,  we may  utilize \eqref{asymptotic of AV 2}, with a bound similar to \eqref{asymptotic of AV 3} in the proof of Theorem \ref{thm:single_mixing}, to conclude that there is a $C_2>0$ so that for $n$ large enough we have 
\begin{align}\label{asymptotic of AV 4}
\mathcal AV(x_n) \le C_2 \lambda_{y_1\to y_1'}(x_n)\ln{\left(\frac{(x_n\vee 1)^{y'_1}}{(x_n\vee 1)^{y_1}}\right)} \le  -C_2 \lambda_{y_1\to y_1'}(x_n),
\end{align}
where the second inequality holds since $\ln{\left(\dfrac{(x_n\vee 1)^{y'_1}}{(x_n\vee 1)^{y_1}}\right)} \to -\infty$, as $n \to \infty$.
Since we know that $2S_i \in T^{D,1}_{\{x_n\}}$, we can then conclude that there is a $C>0$ so that 
\begin{align*}
    \mathcal AV(x_n) \le -C_2 \lambda_{y_1\to y_1'}(x_n)  \le  -C x_{n,i}^2.
\end{align*}
This was our goal, so we are done.
\end{proof}

\section{Extension of Theorem \ref{thm:single_mixing} for exponential ergodicity}
\label{sec:generalizations}

In this section, we generalize Theorem \ref{thm:single_mixing} to provide four more classes of reaction networks for which the associated Markov chains with intensity functions \eqref{mass} are exponentially ergodic. Note that we no longer have previous results guaranteeing  the models under consideration in this section are positive recurrent.  Hence, existence of $\pi$ is now part of each result.  Note, however, that this fact follows immediately since the existence of positive $a$ and $b$ for which $\mathcal{A} V(x) \le -aV(x) + b$ outside a compact immediately implies that $\mathcal{A}V(x) \le -1$ outside that same compact (since $V$ is positive and bounded from below).

The first generalization arises by noting that the proof of Theorem \ref{thm:single_mixing} only utilizes the presence of the outflow reactions.  Thus, the theorem  can be generalized to only allow for a portion of (or even none of)  the inflow reactions.  However, in this case the system is not open and so $\Z^d_{\ge 0}$ may not be a closed communication class, and so the generalized statement  must restrict to a particular closed class.

 \begin{cor}\label{cor:only out flows}
 Let $(\S,\C, \Re)$ be a weakly reversible, binary reaction network that has a single linkage class and $\S=\{S_1,\dots,S_d\}$.   Let $\widetilde{\C} = \C \cup\{\emptyset\}\cup\{S_i \ | \ S_i \in \S\}$ and $\widetilde \Re  = \Re \cup_{S_i\in \S} \{ S_i\to \emptyset\} \cup_{S_i\in \mathcal{I}} \{\emptyset \to S_i\},$ where $\mathcal{I}\subset \S$ is any subset of the species (including the empty set). Let $\K=\{\kappa_{y\to y'}\}$ be a choice of reaction rate constants.  
 Let $X(0) = x \in   \Z^d_{\ge 0}$ be in a closed communication class, $\mathbb{S}$.  Then there is a stationary distribution $\pi$ on that class.  Moreover,
 there exist  $C,\eta>0$ so that the Markov process $X$ with intensity functions \eqref{mass} associated to the reaction network $(\S,\widetilde{\C},\widetilde \Re)$ satisfies
\begin{align*}
\Vert P^t(x,\cdot) - \pi(\cdot) \Vert_{\TV} \le C(|x|+1)\ln(|x|+2) e^{-\eta t}, \quad \text{for all } t \ge 0,
\end{align*}
where  neither $C$ nor $\eta$ depend upon $x\in \mathbb{S}$.
  Thus,  the  mixing time  satisfies $\tau_x^\varepsilon = O(\ln(|x|))$, as $|x|\to \infty$.  
 
 Suppose now that there is a $w\in \R^{d}_{>0}$ for which $w\cdot (y'-y)=0$ for all $y\to y'\in \Re$.  Let $X(0) = x \in   \Z^d_{\ge 0}$ be in a closed communication class, $\mathbb{S}$.   Then there is a stationary distribution $\pi$ on that class.  Moreover,
 there exist  $C,\eta>0$ so that the Markov process $X$ with intensity functions \eqref{mass} associated to the reaction network $(\S,\widetilde{\C},\widetilde \Re)$ satisfies
\begin{align*}
\Vert P^t(x,\cdot) - \pi(\cdot) \Vert_{\TV} \le C(|x|+1) e^{-\eta t}, \quad \text{for all } t \ge 0,
\end{align*}
where neither $C$ nor $\eta$ depend upon $x \in \mathbb{S}$.
Thus, the mixing time  satisfies $\tau_x^{\varepsilon} = O(\ln(|x|))$, as $|x|\to \infty$.
 \end{cor}

\begin{example} Consider  the system with reaction graph in \eqref{eq:example_model1} except $C \to \emptyset$ is removed, i.e. the network with reaction graph
 \begin{align*}
     &\hspace{.33in}C \\
    &\hspace{.3in}\uparrow \\
    &\hspace{.35in} \emptyset \\
    & \nearrow\hspace{-.07in}\swarrow \hspace{.2in} \nwarrow\hspace{-0.07in}\searrow\\
        &A \xrightarrow{\ \ \ \ \ \ \ \ \ } B\\
&  \   \nwarrow \hspace{.3in} \swarrow   \\
& \hspace{0.3in} 2C.
 \end{align*}
 For this model $\Z^d_{\ge 0}$ is irreducible, and there is a unique stationary distribution on $\Z^d_{\ge0}$.  Moreover, because of Corollary \ref{cor:only out flows}, it is also exponentially ergodic with a mixing time that is bounded above by the logarithm of the initial counts. \hfill $\triangle$
\end{example}

As noted in Section \ref{sec:proofs}, the results of Lemma \ref{lemma:tier_topology} are crucial in proving Theorem \ref{thm:single_mixing}. Hence it is possible to find other  structural conditions for the reaction network that guarantee exponential ergodicity by ensuring  Lemma \ref{lemma:tier_topology} holds. Essentially, in Lemma \ref{lemma:tier_topology} the existence of a directed path of reactions starting with a complex $y \in \Td1$ and ending with a complex $y'\not \in \Td1$ was crucial. Hence we can generalize Theorem \ref{thm:single_mixing} with alternative conditions that guarantee the existence of such a path of reactions. 

Using the observation above, the corollary below allows us to drop the single linkage class assumption.   However, $\pi$ could be a point mass (for example, when the network is $2S_1 \to S_1 \to \emptyset$ and no inflow of $S_1$ is added).

\begin{cor}\label{cor: generalization with path}
    Let $(\S,\C, \Re)$ be a binary reaction network with $\S=\{S_1,\dots,S_d\}$. Suppose that for each binary complex $y \in \C$, there is a directed path within the reaction graph beginning with $y$ and ending with either a unary complex (of the form $S_j$) or  the zero complex, $\emptyset$. Let $\widetilde{\C} = \C \cup\{\emptyset\}\cup\{S_i \ | \ S_i \in \S\}$ and $\widetilde \Re  = \Re \cup_{S_i\in \S} \{ S_i\to \emptyset\} \cup_{S_i\in \mathcal{I}} \{\emptyset \to S_i\},$ where $\mathcal{I}\subset \S$ is any subset of the species.  Let $\K=\{\kappa_{y\to y'}\}$ be a choice of reaction rate constants.  
 Let $X(0) = x \in   \Z^d_{\ge 0}$ be in a closed communication class, $\mathbb{S}$.   Then there is a stationary distribution $\pi$ on that class.  Moreover, there
 exist  $C,\eta>0$ so that the Markov process $X$ with intensity functions \eqref{mass} associated to the reaction network $(\S,\widetilde{\C},\widetilde \Re)$ satisfies
\begin{align*}
\Vert P^t(x,\cdot) - \pi(\cdot) \Vert_{\TV} \le C(|x|+1)\ln(|x|+2) e^{-\eta t}, \quad \text{for all } t \ge 0,
\end{align*}
where  neither $C$ nor $\eta$ depend upon $x\in \mathbb{S}$.
  Thus, the  mixing time  satisfies $\tau_x^\varepsilon = O(\ln(|x|))$, as $|x|\to \infty$.  
 
 Suppose now that there is a $w\in \R^{d}_{>0}$ for which $w\cdot (y'-y)=0$ for all $y\to y'\in \Re$.  Let $X(0) = x \in   \Z^d_{\ge 0}$ be in a closed communication class, $\mathbb{S}$.   Then there is a stationary distribution $\pi$ on that class.  Moreover,
  there exist  $C,\eta>0$ so that the Markov process $X$ with intensity functions \eqref{mass} associated to the reaction network $(\S,\widetilde{\C},\widetilde \Re)$ satisfies
\begin{align*}
\Vert P^t(x,\cdot) - \pi(\cdot) \Vert_{\TV} \le C(|x|+1) e^{-\eta t}, \quad \text{for all } t \ge 0,
\end{align*}
where neither $C$ nor $\eta$ depend upon $x \in \mathbb{S}$.
Thus, the mixing time    satisfies $\tau_x^{\varepsilon} = O(\ln(|x|))$, as $|x|\to \infty$.
\end{cor}
\begin{proof}
    It suffices to show that the results (i), (ii), and (iii) in Lemma \ref{lemma:tier_topology} hold. However, in the proof of Lemma \ref{lemma:tier_topology}, the weak reversibility and single linkage class assumptions of $(\S,\C,\Re)$ were only used to show  the existence of a directed path of reactions beginning with some complex $y\in \Td1$ and ending with either a unary complex or the zero complex when $\Td1$ only contains  binary complexes. Hence by directly assuming the existence of such a sequence, the result follows.
\end{proof}

\begin{example}
    Let a reaction network $(\S,\C,\Re)$ be described with the following reaction graph.
    \begin{align*}
        2B \to A+B \to C, \quad 2C\to A \leftarrow \emptyset \to B.
    \end{align*}
    Then each binary complex in $\C$ has a directed path of reactions beginning with itself and ending with a unary complex. Then adding all the out-flows, we have the new reaction network $(\S,\widetilde \C, \widetilde \Re)$, which has reaction graph
    \begin{align*}
            2B \to A+B \to \ \ &C\\
            & \hspace{-0.11cm} \downarrow\\
            \quad 2C\to A \rightleftharpoons \ &\emptyset \rightleftharpoons B.
    \end{align*}
    By Collorary \ref{cor: generalization with path}, for any choice of rate constants the Markov model associated to the network $(\S,\widetilde \C, \widetilde \Re)$ is exponentially ergodic with mixing time bounded above by the logarithm of the initial counts.  \hfill $\triangle$
\end{example}

We may also generalize by not requiring that all the species have an out-flow reaction ($S_i \to \emptyset$). However, in this case we add the assumption that there is a ``chain'' of first order reactions (of the form $S_i \to S_j$) connecting a species with no out-flow reaction to one with an out-flow reaction. We first consider the case where only one species, $S_1$, say, has an outflow.

\begin{cor}\label{cor: generalization with unary path}
    Let $(\S,\C, \Re)$ be a binary reaction network with $\S=\{S_1,\dots,S_d\}$ and $\{S_1,S_2,\dots,S_d\} \subset \C$. Suppose that
    \begin{enumerate}
        \item for each binary complex $y \in \C$, there is a directed path within the reaction graph beginning with $y$ and ending with either a unary complex (of the form $S_j$) or  the zero complex, $\emptyset$.
        \item for each unary complex  $S_i \in \C$ such that $i\neq 1$, there is a directed path within the reaction graph such that it begins with $S_i$, ends with $S_1$, and consists only of unary complexes (i.e. the path is of the form $S_i\to S_{i_1}\to S_{i_2}\to \dots \to S_1$).
    \end{enumerate}
    Let $\widetilde{\C} = \C \cup \{\emptyset\}$ and $\widetilde \Re  = \Re \cup \{S_1\to \emptyset\} \cup_{S_i\in \mathcal{I}} \{\emptyset \to S_i\},$ where $\mathcal{I}\subset \S$ is any subset of the species.  Let $\K=\{\kappa_{y\to y'}\}$ be a choice of reaction rate constants.  
 Let $X(0) = x \in   \Z^d_{\ge 0}$ be in a closed communication class, $\mathbb{S}$. Then there is a stationary distribution $\pi$ on that class.  Moreover, 
 there exist  $C,\eta>0$ so that the Markov process $X$ with intensity functions \eqref{mass} associated to the reaction network $(\S,\widetilde{\C},\widetilde \Re)$ satisfies
\begin{align*}
\Vert P^t(x,\cdot) - \pi(\cdot) \Vert_{\TV} \le C(|x|+1)\ln(|x|+2) e^{-\eta t}, \quad \text{for all } t \ge 0,
\end{align*}
where  neither $C$ nor $\eta$ depend upon $x\in \mathbb{S}$.
  Therefore the  mixing time of the model satisfies $\tau_x^\varepsilon = O(\ln(|x|))$, as $|x|\to \infty$.  
 
 Suppose now that there is a $w\in \R^{d}_{>0}$ for which $w\cdot (y'-y)=0$ for all $y\to y'\in \Re$.  Let $X(0) = x \in   \Z^d_{\ge 0}$ be in a closed communication class, $\mathbb{S}$.  Then there is a stationary distribution $\pi$ on that class.  Moreover,
 there exist  $C,\eta>0$ so that the Markov process $X$ with intensity functions \eqref{mass} associated to the reaction network $(\S,\widetilde{\C},\widetilde \Re)$ satisfies
\begin{align*}
\Vert P^t(x,\cdot) - \pi(\cdot) \Vert_{\TV} \le C(|x|+1) e^{-\eta t}, \quad \text{for all } t \ge 0,
\end{align*}
where neither $C$ nor $\eta$ depend upon $x \in \mathbb{S}$.
The mixing time in this case still satisfies $\tau_x^{\varepsilon} = O(\ln(|x|))$, as $|x|\to \infty$.
\end{cor}
\begin{proof}
    We first show that the results in Lemma \ref{lemma:tier_topology} hold for $(\S,\widetilde \C, \widetilde \Re)$. Let $\{x_n\}$ be an arbitrary tier-sequence. We only need to consider the case that at least one unary complex is in $\Td1$. The proof of the other cases is the same as the proof of Lemma \ref{lemma:tier_topology}. Let $y=S_{\ell} \in \Td1$. Then on a direct path $S_{\ell}\to S_{i_1}\to \cdots \to S_1\to \emptyset$, there must exist a reaction of the form either $S_j\to S_k$ for some $j\neq k$ such that $S_j \in \Td1$ and $S_k \not \in \Td1$ or $S_1\to \emptyset$ with $S_1 \in \Td1$ and $\emptyset \not \in \Td1$ (as always). Then letting this reaction be $y_1\to y_1'$, Lemma \ref{lemma:tier_topology} follows.

    To complete the proof, it suffices to start at \eqref{asymptotic of AV 3} because the previous parts in the proof of Theorem \ref{thm:single_mixing} stand without further modification. Let $\{x_n\}$ be an arbitrary tier-sequence. Note that i) if there exists a reaction $y\to y'$ such that $y=S_r$ for some $r$ and $y \succ_D y'$ and ii) if $x_{n,r}$ is a maximal coordinate of $\{x_n\}$, then the proof is done in the same way as the proof of Theorem \ref{thm:single_mixing}.  Let $x_{n,\ell}$ be the maximal coordinate of $\{x_n\}$.   
    Then over a directed path $S_\ell=S_{i_0} \to S_{i_1}\to S_{i_2} \to \cdots \to S_1 \to \emptyset$, we have that either i) there exists $S_r \to S_m \in \Re$ for some $r, m$ such that $x_{n,r}$ is a maximal coordinate and $S_r \succ_D S_m$ or ii) $S_1 \to \emptyset$ satisfies that $x_{n,1}$ is a maximal coordinate and $S_1\succ_D \emptyset$. This completes the proof. 
\end{proof}

\begin{example}
Let $(\S,\C,\Re)$ be a reaction network that is described with the following reaction graph
\begin{align*}
            & \hspace{-3.2cm} \quad 2B \xrightarrow{\ \ \ \ \ } A+B \xrightarrow{ \ \ \ \ \ } \ \  C\\[-1ex]
            & \ \nearrow \ \ \ \  \ \ \ \searrow\\[-1ex]
             \quad 2C\xrightarrow{ \ \ \ \ \ } \ &A  \xleftarrow{ \ \ \ \ \ \ \ \ \ \ }  B\\[-1ex]
            \displaystyle &  \ \ \nwarrow \quad \qquad  \searrow \hspace{-0.2cm} \nwarrow\\[-1ex]
             &\qquad \  \emptyset \xrightarrow{\ \ \ \ \ \ \  \ \ } D.
    \end{align*}
    
 Then $(\S, \C, \Re)$ satisfies conditions (i) and (ii) of Corollary \ref{cor: generalization with unary path} where $S_1$ can be taken to be $A$, $B$, or $C$.  Hence, $(\S,\widetilde \C,\widetilde \Re)$ can be made with an additional one of the out-flows $\{A\to \emptyset, B\to \emptyset, C\to \emptyset, D\to \emptyset\}$. 
 Then, by Corollary \ref{cor: generalization with unary path}, for any choice of rate constants the Markov model associated to the network $(\S,\widetilde \C, \widetilde \Re)$ is exponentially ergodic with mixing time bounded above by the logarithm of the initial counts. \hfill $\triangle$ 
\end{example}

 Finally,  Corollary \ref{cor: generalization with unary path} can be generalized further by adding   additional out-flows   to $(\S,\C,\Re)$ when some unary complex does not satisfy the directed path condition.  In the corollary below, $\S_p$ plays the role of $\{S_2,\dots,S_d\}$ and $\S_e$ plays the role of $\{S_1\}$ in Corollary \ref{cor: generalization with unary path}.

\begin{cor}\label{cor: generalization with unary path2}
    Let $(\S,\C, \Re)$ be a binary reaction network with $\S=\{S_1,\dots,S_d\}$. 
     We assume the following conditions. 
     \begin{enumerate}
         \item     For each binary complex $y \in \C$, there is a directed path within the reaction graph beginning with $y$ and ending with either a unary complex (of the form $S_j$) or  the zero complex, $\emptyset$.         
        \item There exist $\S_p \subseteq \S$ and a non-empty subset $\S_e \subseteq \S$ such that (i) $\S = \S_p \cup \S_e$, and (ii) for each $S_i\in \S_p$ there is a directed path within the reaction graph such that it begins with $S_i$ and ends with some $S_k \in \S_e$. Furthermore the directed path consists only of unary complexes (i.e. the path is of the form $S_i\to S_{i_1}\to S_{i_2}\to \dots \to S_k$).  
        
     \end{enumerate}
      Let $\widetilde{\C} = \C  \cup \{S_i \ | \ S_i \in \S_e \} \cup\{\emptyset\}$ and $\widetilde \Re  = \Re \cup_{S \in \S_e} \{S\to \emptyset\}\cup_{S_i\in \mathcal{I}} \{\emptyset \to S_i\},$ where $\mathcal{I}\subset \S$ is any subset of the species.  Let $\K=\{\kappa_{y\to y'}\}$ be a choice of reaction rate constants.  Let $X(0) = x \in   \Z^d_{\ge 0}$ be in a closed communication class. Then there is a stationary distribution $\pi$ on that class.  Moreover,
      there exist  $C,\eta>0$ so that the Markov process $X$ with intensity functions \eqref{mass} associated to the reaction network $(\S,\widetilde{\C},\widetilde \Re)$ satisfies
\begin{align*}
\Vert P^t(x,\cdot) - \pi(\cdot) \Vert_{\TV} \le C(|x|+1)\ln(|x|+2) e^{-\eta t}, \quad \text{for all } t \ge 0,
\end{align*}
where  neither $C$ nor $\eta$ depend upon $x\in \mathbb{S}$.
  Therefore the  mixing time of the model satisfies $\tau_x^\varepsilon = O(\ln(|x|))$, as $|x|\to \infty$.  
 
 Suppose now that there is a $w\in \R^{d}_{>0}$ for which $w\cdot (y'-y)=0$ for all $y\to y'\in \Re$.  Let $X(0) = x \in   \Z^d_{\ge 0}$ be in a closed communication class, $\mathbb{S}$. Then there is a stationary distribution $\pi$ on that class.  Moreover,
 there exist  $C,\eta>0$ so that the Markov process $X$ with intensity functions \eqref{mass} associated to the reaction network $(\S,\widetilde{\C},\widetilde \Re)$ satisfies
\begin{align*}
\Vert P^t(x,\cdot) - \pi(\cdot) \Vert_{\TV} \le C(|x|+1) e^{-\eta t}, \quad \text{for all } t \ge 0,
\end{align*}
where neither $C$ nor $\eta$ depend upon $x \in \mathbb{S}$.
The mixing time in this case still satisfies $\tau_x^{\varepsilon} = O(\ln(|x|))$, as $|x|\to \infty$.
\end{cor}
\begin{proof}
   If $\S_e\neq \empty$ and $\S_p\neq \emptyset$, the proof is exactly the same as the proof of Corollary \ref{cor: generalization with unary path} after replacement of $S_1$ by $S_{k} \in \S_e$.
    If $\S_p=\emptyset$, then $\S_p^c=\S$. In this case the result still follows since  $(\S,\widetilde \C,\widetilde \Re)$ has the same condition of Corollary \ref{cor: generalization with path}.
\end{proof}

\begin{example}\label{example:kdjfakl;j}
    We give an example for Corollary \ref{cor: generalization with unary path2} with a version of enzyme-substrate kinetics.
    Consider the reaction network described by the reaction graph below:
    \begin{align}\label{eq:es kinetics}
        &S \rightleftharpoons P \leftarrow \emptyset \to  E\\ \notag
        &\hspace{1.7cm}\downarrow \hspace{-0.05cm }\uparrow\\
        &S+E \rightleftharpoons  \ SE \rightleftharpoons E+P. \notag
    \end{align}
    We put $\S_p=\{S\}$ and $\S_e=\{P, E, SE\}$. 
    Then $(\S,\widetilde\C, \widetilde \Re)$ is made by adding reactions $E\to \emptyset$ and $P\to \emptyset$ to this reaction network (note that $SE\to \emptyset$ is already in the network and so does not need to be added).
    Similar to Figure \ref{fig1}, in Figure \ref{fig2} we provide the results of a numerical simulation for the continuous-time Markov chain $X$ associated with $(\S,\widetilde\C, \widetilde \Re)$ in which all rate constants are selected to be equal to one. 
    In this case, the stationary distribution of $X$ is a product form of Poissons \cite{AndProdForm}. \hfill $\triangle$
\end{example}

\begin{figure}
    \centering
    \includegraphics[width=0.9\textwidth]{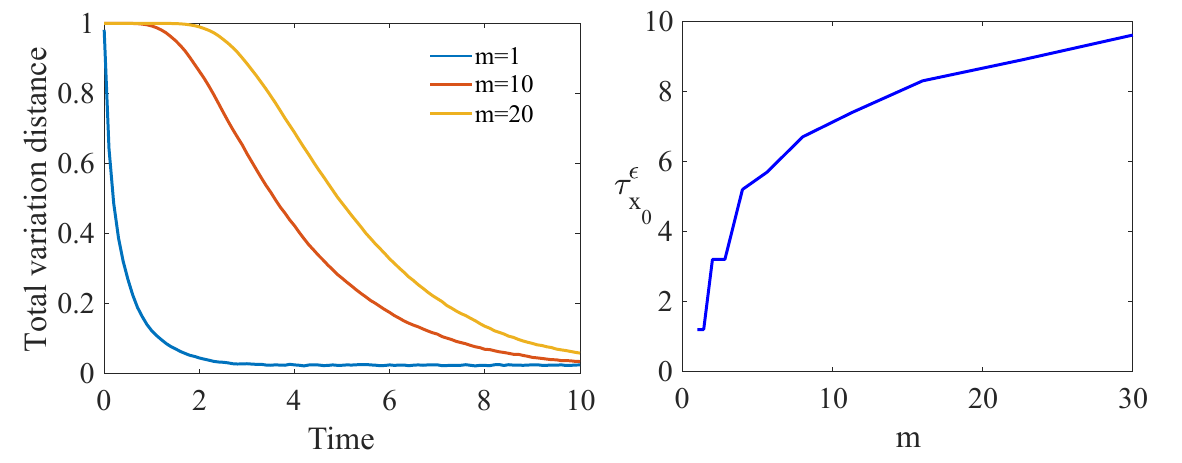}
    \caption{Plots for the model provided in Example \ref{example:kdjfakl;j} in \eqref{eq:es kinetics}. 
 For each initial condition $X(0)=(m,m,m)^\top$, the image on the left provides  a plot of  $\Vert P^t(x,\cdot)-\pi(\cdot)\Vert_{\TV}$ over time for the reaction network \eqref{eq:es kinetics}, where $\pi$ represents the stationary distribution. The image on the right provides a plot  of the mixing times $\tau^\varepsilon_{x_m}$ with $\varepsilon=0.1$ and $x_m=(m,m,m)^\top$, as a function of $m$.  Note that the growth on the right appears to be logarithmic, which agrees with our theory.
    }
    \label{fig2}
\end{figure}

\section*{Acknowledgments}

DFA is supported from NSF-DMS-2051498 and Army Research Office grant W911NF-18-1-0324. JK is supported from the National Research Foundation of Korea (NRF) grant funded by the Korea government (MSIT)(No. 2022R1C1C1008491) and the Basic Science Research Institute Fund, whose NRF grant number is 2021R1A6A1A10042944.  We also thank the three reviewers of our paper, whose comments improved the final product appreciably.

\section*{Conflict of interest}

The authors declare there is no conflict of interest.

\bibliographystyle{plain}
\bibliography{bibliographyMixingTimeOrig}

\begin{thebibliography}{10}

\bibitem{AM202}
Andrea Agazzi and Jonathan~C. Mattingly.
\newblock Seemingly stable chemical kinetics can be stable, marginally stable,
  or unstable.
\newblock {\em Commun. Math. Sci.}, 18(6):1605--1642, 2020.

\bibitem{Anderson2007a}
David~F. Anderson.
\newblock {A modified Next Reaction Method for simulating chemical systems with
  time dependent propensities and delays}.
\newblock {\em J. Chem. Phys.}, 127(21):214107, 2007.

\bibitem{Anderson2007b}
David~F. Anderson.
\newblock {Incorporating postleap checks in tau-leaping}.
\newblock {\em J. Chem. Phys.}, 128(5):54103, 2008.

\bibitem{AndBounded_ONE2011}
David~F. Anderson.
\newblock {Boundedness of trajectories for weakly reversible, single linkage
  class reaction systems}.
\newblock {\em Journal of Mathematical Chemistry}, 49(10):2275--2290, 2011.

\bibitem{AndGAC_ONE2011}
David~F. Anderson.
\newblock A proof of the global attractor conjecture in the single linkage
  class case.
\newblock {\em SIAM J. Appl. Math}, 71(4):1487 -- 1508, 2011.

\bibitem{anderson2020stochastically}
David~F. Anderson, Daniele Cappelletti, and Jinsu Kim.
\newblock Stochastically modeled weakly reversible reaction networks with a
  single linkage class.
\newblock {\em Journal of Applied Probability}, 57(3):792--810, 2020.

\bibitem{anderson2020tier}
David~F. Anderson, Daniele Cappelletti, Jinsu Kim, and Tung~D. Nguyen.
\newblock Tier structure of strongly endotactic reaction networks.
\newblock {\em Stochastic Processes and their Applications},
  130(12):7218--7259, 2020.

\bibitem{anderson2016product}
David~F. Anderson and Simon~L. Cotter.
\newblock Product-form stationary distributions for deficiency zero networks
  with non-mass action kinetics.
\newblock {\em Bulletin of mathematical biology}, 78(12):2390--2407, 2016.

\bibitem{anderson2015lyapunov}
David~F. Anderson, Gheorghe Craciun, Manoj Gopalkrishnan, and Carsten Wiuf.
\newblock Lyapunov functions, stationary distributions, and non-equilibrium
  potential for reaction networks.
\newblock {\em Bulletin of mathematical biology}, 77(9):1744--1767, 2015.

\bibitem{AndProdForm}
David~F. Anderson, Gheorghe Craciun, and Thomas~G. Kurtz.
\newblock {Product-form stationary distributions for deficiency zero chemical
  reaction networks}.
\newblock {\em Bull. Math. Biol.}, 72(8):1947--1970, 2010.

\bibitem{AndHigham2012}
David~F. Anderson and Desmond~J. Higham.
\newblock {Multi-level {M}onte {C}arlo for continuous time {M}arkov chains,
  with applications in biochemical kinetics}.
\newblock {\em SIAM: Multiscale Modeling and Simulation}, 10(1):146--179, 2012.

\bibitem{anderson2018some}
David~F. Anderson and Jinsu Kim.
\newblock Some network conditions for positive recurrence of stochastically
  modeled reaction networks.
\newblock {\em SIAM Journal on Applied Mathematics}, 78(5):2692--2713, 2018.

\bibitem{AndKurtz2011}
David~F. Anderson and Thomas~G. Kurtz.
\newblock {Continuous time Markov chain models for chemical reaction networks}.
\newblock In H~Koeppl Et~al., editor, {\em Design and Analysis of Biomolecular
  Circuits: Engineering Approaches to Systems and Synthetic Biology}, pages
  3--42. Springer, 2011.

\bibitem{AK2015}
David~F. Anderson and Thomas~G. Kurtz.
\newblock {\em {Stochastic analysis of biochemical systems}}, volume 1.2 of
  {\em Stochastics in Biological Systems}.
\newblock Springer International Publishing, Switzerland, 1 edition, 2015.

\bibitem{superlya2012}
Avanti Athreya, Tiffany Kolba, and Jonathan~C. Mattingly.
\newblock {Propagating Lyapunov functions to prove noise-induced
  stabilization}.
\newblock {\em Electron J. Probab.}, 17(96):1--38, 2012.

\bibitem{clark1991first}
John Clark and Derek~Allan Holton.
\newblock {\em A first look at graph theory}.
\newblock World Scientific, 1991.

\bibitem{ethier2009markov}
Stewart~N. Ethier and Thomas~G. Kurtz.
\newblock {\em Markov processes: characterization and convergence}, volume 282.
\newblock John Wiley \& Sons, 2009.

\bibitem{Feinberg72}
Martin Feinberg.
\newblock {Complex balancing in general kinetic systems}.
\newblock {\em Arch. Rational Mech. Anal.}, 49:187--194, 1972.

\bibitem{Feinberg87}
Martin Feinberg.
\newblock {Chemical reaction network structure and the stability of complex
  isothermal reactors - {I}. The Deficiency Zero and Deficiency One theorems,
  Review Article 25}.
\newblock {\em Chem. Eng. Sci.}, 42(10):2229--2268, 1987.

\bibitem{Othmer2005}
Chetan Gadgil, Chang~Hyeong Lee, and Hans~G. Othmer.
\newblock {A stochastic analysis of first-order reaction networks}.
\newblock {\em Bull. Math. Bio.}, 67:901--946, 2005.

\bibitem{Gibson2000}
Michael~A. Gibson and Jehoshua Bruck.
\newblock {Efficient exact stochastic simulation of chemical systems with many
  species and many channels}.
\newblock {\em J. Phys. Chem. A}, 105:1876--1889, 2000.

\bibitem{Gill76}
Daniel~T. Gillespie.
\newblock {A general method for numerically simulating the stochastic time
  evolution of coupled chemical reactions}.
\newblock {\em J. Comput. Phys.}, 22:403--434, 1976.

\bibitem{Gill77}
Daniel~T. Gillespie.
\newblock {Exact Stochastic Simulation of Coupled Chemical Reactions}.
\newblock {\em J. Phys. Chem.}, 81(25):2340--2361, 1977.

\bibitem{Gill2001}
Daniel~T. Gillespie.
\newblock {Approximate accelerated simulation of chemically reaction systems}.
\newblock {\em J. Chem. Phys.}, 115(4):1716--1733, 2001.

\bibitem{hoessly2020algebraic}
Linard Hoessly and Beatriz Pascual-Escudero.
\newblock An algebraic approach to product-form stationary distributions for
  some reaction networks.
\newblock {\em arXiv preprint arXiv:2012.03227}, 2020.

\bibitem{HornJack72}
Friedrich Josef~Maria Horn and Roy Jackson.
\newblock {General Mass Action Kinetics}.
\newblock {\em Arch. Rat. Mech. Anal.}, 47:81--116, 1972.

\bibitem{YuvalLevinMixing}
David~A. Leven and Yuval Peres.
\newblock {\em {Markov Chains and Mixing Times}}.
\newblock American Mathematical Society, 2009.

\bibitem{MT-LyaFosterIII}
Sean~P. Meyn and Richard~L. Tweedie.
\newblock {Stability of Markovian Processes III : Foster-Lyapunov Criteria for
  Continuous-Time Processes}.
\newblock {\em Advances in Applied Probability}, 25(3):518--548, 1993.

\bibitem{NorrisMC97}
James Norris.
\newblock {\em {Markov Chains}}.
\newblock Cambridge University Press, 1997.

\bibitem{Tweedie1981}
Richard~L. Tweedie.
\newblock {Criteria for ergodicity, exponential ergodicity and strong
  ergodicity of Markov processes}.
\newblock {\em J. Appl. Prob.}, 18(1):122--130, 1981.

\bibitem{Wilkinson2006}
Darren~J. Wilkinson.
\newblock {\em {Stochastic Modelling for Systems Biology}}.
\newblock Chapman and Hall/CRC Press, 2006.

\bibitem{XuHansenWiuf2022}
Chuang Xu, Mads~C. Hansen, and Carsten Wiuf.
\newblock Full classification of dynamics for one-dimensional continuous time
  markov chains with polynomial transition rates.
\newblock arXiv:2006.10548, 2021.

\end{thebibliography}

\end{document}